\documentclass[11pt]{amsart}
\usepackage{epsfig,color}
\usepackage{blindtext}
\usepackage{graphicx}
\usepackage{enumitem}
\usepackage{url}
\usepackage{amssymb}
\usepackage{graphicx,import}
\usepackage{comment}
\usepackage{esint}
\usepackage{xypic}

\usepackage[margin = 1.4in] {geometry}

\usepackage{hyperref}

\theoremstyle{theorem}

\newtheorem{theorem}{Theorem}[section]

\newtheorem{proposition}[theorem]{Proposition}
\newtheorem{lemma}[theorem]{Lemma}

\theoremstyle{definition}

\newtheorem{definition}[theorem]{Definition}
\newtheorem*{remark*}{Remark}

\numberwithin{equation}{section}

\newcommand{\R}{\mathbb{R}}

\newcommand{\N}{\mathbb{N}}

\DeclareMathOperator{\Ric}{Ric}

\DeclareMathOperator{\diam}{diam}
\newcommand{\eps}{\varepsilon}

\DeclareMathOperator{\vol}{Vol} 

\title{Volume growth of 3-manifolds with scalar curvature lower bounds}

\author{Otis Chodosh}
\address{Department of Mathematics, Stanford University, Building 380, Stanford, CA 94305, USA}
\email{ochodosh@stanford.edu}

\author{Chao Li}
\address{Courant Institute, New York University, 251 Mercer St, New York, NY 10012, USA}
\email{chaoli@nyu.edu}

\author{Douglas Stryker}
\address{Department of Mathematics, Princeton University, Fine Hall, 304 Washington Road, Princeton, NJ 08540, USA}
\email{dstryker@princeton.edu}

\begin{document}

\maketitle

\begin{abstract}
We give a new proof of a recent result of Munteanu--Wang relating scalar curvature to volume growth on a $3$-manifold with non-negative Ricci curvature. Our proof relies on the theory of $\mu$-bubbles introduced by Gromov as well as the almost splitting theorem due to Cheeger--Colding. 
\end{abstract}

\section{Introduction}

In this note we give a new proof of (and slightly generalize) the following volume growth estimate recently proven by Munteanu--Wang \cite[Theorem 5.6]{MW:2}. 

\begin{theorem}\label{thm:main}
Let $(M^3, g)$ be a complete noncompact 3-manifold with $\Ric_g \geq 0$. Then
\begin{equation}\label{eq:scal_growth}
\liminf_{d(x_0, x) \to \infty} R_g(x) \leq C(x_0, M, g) < +\infty
\end{equation}
for all $x_0 \in M$. Moreover, if
\begin{equation}\label{eq:scal_bound}
R_g(x) \geq d(x, x_0)^{-\alpha}
\end{equation}
outside a compact set $K$ for some $x_0 \in M$ and $0 \leq \alpha < 2$, then
\begin{equation}\label{eq:vol_growth}
\vol(B_r(x_0)) \leq C(x_0,M, g)r^{1+\alpha}
\end{equation}
for all $r > 0$.
\end{theorem}

(We note that Theorem \ref{thm:main} here considers the optimal range\footnote{When $\alpha \geq 2$, any estimate from the scalar curvature inequality \eqref{eq:scal_bound} would be weaker than the Bishop--Gromov cubic volume growth estimate from non-negative Ricci; moreover, by \eqref{eq:scal_growth} it is not possible that \eqref{eq:scal_bound} holds with $\alpha < 0$.} of $\alpha$ while \cite[Theorem 5.6]{MW:2} only proves  Theorem \ref{thm:main} for $0\leq \alpha \leq 1$.\footnote{On the other hand, we point out that \cite[Theorem 5.6]{MW:2} does consider a certain notion of negativity of the Ricci curvature at infinity, so in that respect \cite[Theorem 5.6]{MW:2} is more general than Theorem \ref{thm:main}. Moreover, \cite[Theorem 5.6]{MW:2} yields a more explicit estimate for the constant $C(x_0,M,g)$ than we do here.}) The proof of Theorem \ref{thm:main} given by Munteanu--Wang is based on their analysis of certain harmonic functions on such manifolds (see also their earlier work \cite{MW:1} as well as \cite{CLstable}). Our proof is rather different and instead relies on the theory of $\mu$-bubbles introduced by Gromov \cite{Gromov:metric-inequalities}.

The techniques used here are inspired by our recent article on (non-compact) stable minimal hypersurfaces in $4$-manifolds \cite{CLS:stable} (see also \cite{CL:aniso}). We note that in this paper it is necessary to handle the possibility that $\partial B_r(x_0) \subset M$ may have many connected components (even when $(M,g)$ has only one end). (In \cite{CLS:stable} this issue was avoided since all that was needed was an efficient cutoff function.)  Here we use the almost splitting theorem of Cheeger--Colding to show that even if there are many components of $B_r(x_0)$ they do not contribute too much to the volume growth. 

We note that Theorem \ref{thm:main} is related to well-known conjectures of Yau \cite{Yau:problems} and Gromov \cite{Gromov:large}. Yau has conjectured that if $(M^n,g)$ has $\Ric_g\geq 0$, then $\int_{B_r(x_0)} R_g \leq C r^{n-2}$ for all $r>0$ while Gromov has conjectured that if $(M,g)$ satisfies $\Ric_g\geq 0$ and $R_g(x) \geq 1$ then $\vol(B_r(x_0)) \leq C r^{n-2}$. For some works related to these conjectures we refer to \cite{Petrunin:upperbound,Naber:conj,Xu:Yauconj,Zhu:psc.int}.

\subsection{Acknowledgements}
We are grateful to Jeff Cheeger and Aaron Naber for some helpful discussions. O.C. was supported by a Terman Fellowship and a Sloan Fellowship. C.L. was supported by an NSF grant (DMS-2202343). D.S. was supported by an NDSEG Fellowship.

\section{Proof of main result}

By using the splitting theorem \cite{CheegerGromoll}, it is easy to see that Theorem \ref{thm:main} holds for complete 3-manifolds with $\Ric_g \geq 0$ and two (or \emph{a priori} more) ends, so it suffices to handle the case when $M$ has one end.

A key tool is the following result which is a consequence of the theory of $\mu$-bubbles due to Gromov \cite{Gromov:metric-inequalities}. See, for example, \cite[Lemma 5.3]{CLS:stable} (with references to \cite{CL:aspherical}) for a proof.
\begin{lemma}[$\mu$-bubble diameter bound]\label{lem:mububble}
Let $(N^3, g)$ be a 3-manifold with boundary satisfying $R_g \geq 1$. Then there are universal constants $L > 0$ and $c > 0$ such that if there is a $p \in N$ with $d_N(p, \partial N) > L/2$, then there is an open set $\Omega \subset B_{L/2}(\partial N) \cap N$ and a smooth surface $\Sigma^2$ such that $\partial \Omega = \Sigma \sqcup \partial N$ and each component of $\Sigma$ has diameter at most $c$.
\end{lemma}

Let $L$ and $c$ denote the constants from Lemma \ref{lem:mububble}. Note that we are free to make $L$ larger, so we will assume $L \gg c$. The following is the main geometric result used in the proof of Theorem \ref{thm:main}. 

\begin{lemma}\label{lem:main}
Let $(M^3, g)$ be a complete 3-manifold with $\Ric_g \geq 0$ and one end. Let $x \in M$. There is an $r_0(x, M, g) > 0$ and a universal constant $C > 0$ so that if
\[ R_g\mid_{\overline{B}_{r+a_1+a_2}(x) \setminus B_{r}(x)} \geq 1 \]
for some $r \geq r_0$ and some $a_1, a_2 \in [L, 2L]$, then
\[ \vol(\overline{E}_1 \setminus E_2) \leq C, \]
where $E_1$ and $E_2$ are the unique unbounded components of $M \setminus \overline{B}_{r + a_1}(x)$ and $M \setminus \overline{B}_{r + a_1+a_2}(x)$ respectively.
\end{lemma}
\begin{proof}
Let $E_0$ be the unique unbounded component of $M \setminus \overline{B}_{r}(x)$, and $E_1$, $E_2$ as in the statement. By \cite[Corollary 1.5]{Anderson} (or \cite{SY:ric,Liu:milnor}), we have $b_1(M) < \infty$. Then by \cite[Proposition 3.2]{CLS:stable}, $\partial E_k$ are connected so long as $r\geq r_0$ for $r_0(x, M, g) > 0$ fixed. 

Let $\gamma: \R_+ \to M$ be a geodesic ray associated with the unique end of $M$. For all $r>r_0$, $\gamma\cap \partial B_r(x)$ lies on the boundary of the unique unbounded component of $M\setminus B_r(x)$. We apply Lemma \ref{lem:mububble} to $E_k$ for $k = 0$, 1. As in the proof of \cite[Lemma 5.4]{CLS:stable}, we obtain a connected surface $\Sigma_k$ in $B_{L/2}(\partial E_k) \cap E_k$ with $\diam(\Sigma_k) \leq c$ that separates $\partial E_k$ from the end of $M$. Choose $t_k \in \R_+$ with $\gamma(t_k) \in\Sigma_k$, for $k=0,1$. Note that $d_g(\gamma(t_0),\gamma(t_1)) \leq a_1 + a_2 \leq 4L$ so $d_g(x_0, x_1) \leq 4L + 2c =: D$ for any $x_0 \in \Sigma_0$ and $x_1 \in \Sigma_1$.

We now define some constants. Let  $b=3c, A=2\sqrt{b^2+D^2}$. Fix $R\in \R$ such that
\begin{equation}\label{eq.assumption.R}
R\ge A+4L,\quad \sqrt{b^2+(2R+D)^2}+1 <2D +2R.
\end{equation}
Then take $\delta \in (0,1)$ such that
\begin{equation}\label{eq.assumption.delta}
\delta<\sqrt{b^2+D^2},\quad 14\delta + 6\sqrt{\delta (D+R)}<c,\quad 2c+22 \delta +6\sqrt{\delta (D+R)}<\frac L2.
\end{equation}
Note that all constants here are numerical (i.e., independent of $(M,g)$ and $r_0$). 

By the Cheeger--Colding almost splitting theorem \cite[Theorem 6.62]{CC:splitting}, assuming $r_0$ is sufficiently large depending on $D,R,\delta$, there is a length space $(Y, d)$ with
\[ d_{GH}(B_{D+R}(\gamma(t_1)) \subset (M, d_g),\  B_{D+R}(y, 0) \subset (Y \times \R, d \times d_{\text{Euc}})) < \delta. \]
(See Definition \ref{defi:GH} for the definition of the Gromov--Hausdorff distance.) Below we fix $(Y,d)$ with this property. \\

\noindent \emph{Claim 1}. $\diam(Y, d) \leq b$.
\begin{proof}
By the definition of $D$, we have $\Sigma_0 \cup \Sigma_1 \subset B_{D+R}(\gamma(t_1))$. Since $\Sigma_1$ is connected and separating, and $B_{D+R}(\gamma(t_1)) \supset \Sigma_1$ is connected, $B_{D+R}(\gamma(t_1)) \setminus \Sigma_1$ has two components. Let $\Omega_i$, $i=1,2$, denote the two components. Let
\[ f: B_{D+R}(\gamma(t_1)) \to B_{D+R}(y, 0) \]
be a $\delta$-Gromov--Hausdorff approximation (cf.\ Definition \ref{defi:eps-GH-approx}) given by the almost splitting theorem. Then,
\[ B_{\delta}(f(\Omega_1)) \cup B_{\delta}(f(\Sigma_1)) \cup B_{\delta}(f(\Omega_2)) \]
covers $B_{D+R}(y, 0)$. Let
\[ S := B_{\delta}(f(B_{2\delta}(\Sigma_1))) \subset B_{D+R}(y, 0), \]
and let
\[ \Lambda_i := B_{D+R}(y, 0) \cap B_{\delta}(f(\Omega_i)) \setminus S. \]
Then
\[ B_{D + R}(y, 0) = \Lambda_1 \cup S \cup \Lambda_2. \]

We first show an upper bound for the (extrinsic) diameter of $S$. Suppose $p,\ q \in S$. Then there are $p',\ q' \in B_{2\delta}(\Sigma_1)$ and $p'',\ q'' \in \Sigma_1$ so that
\[ d(p, q) \leq d(f(p'), f(q')) + 2\delta \leq d(p', q') + 3\delta \leq d(p'', q'') + 7\delta \leq c + 7\delta. \]
Hence,
\[ \diam(S) \leq c + 7\delta. \]

Second, we show a lower bound for the diameter of $\Lambda_i$. Note that the length of a component of $\gamma$ in $\Omega_i$ is at least $L/2$. Then we can take $p,\ q \in \Omega_i \setminus B_{5\delta}(\Sigma_1)$ with $d(p, q) \geq L/2 - 5\delta$. For any $x \in B_{2\delta}(\Sigma_1)$, we have
\[ d(f(p), f(x)) \geq d(p, x) - \delta \geq d(p, \Sigma_1) - 3\delta \geq 2\delta \]
(and similarly for $q$). Thus there are $p', q' \in \Lambda_i$ satisfying
\[ d(p', q') \geq d(f(p), f(q)) - 2\delta \geq d(p, q) - 3\delta \geq L/2 - 8\delta. \]
Hence,
\[ \diam(\Lambda_i) \geq L/2 - 8\delta. \]

Third, we show that $\Lambda_1$ and $\Lambda_2$ are separated by a positive distance. Let $p \in \Lambda_1$ and $q \in \Lambda_2$. Then there are $p' \in \Omega_1 \setminus B_{2\delta}(\Sigma_1)$ and $q' \in \Omega_2 \setminus B_{2\delta}(\Sigma_1)$ satisfying (by the fact that $\Sigma_1$ is separating $\Omega_1$ and $\Omega_2$)
\[ d(p, q) \geq d(f(p'), f(q')) - 2\delta \geq d(p', q') - 3\delta \geq \delta. \]
Hence,
\[ d(\Lambda_1, \Lambda_2) \geq \delta. \]

Finally, we show that 
\[ \diam(Y,d) \leq 2c + 14\delta + 6\sqrt{\delta (D+R)}. \]

Suppose otherwise for contradiction. Then $c + 7\delta + 3\sqrt{\delta (D+R)} < \frac{1}{2}\diam(Y, d)$. By Proposition \ref{prop:centers} (using $9\delta \leq D+R$) and the diameter bound for $S$, we have
\[ S \subset B_{c+7\delta + 3\sqrt{\delta (D+R)}}(y, 0). \]
By Proposition \ref{prop:prod_met},
\[ B_{D+R}(y, 0) \setminus B_{c+7\delta + 3\sqrt{\delta (D+R)}}(y, 0) \]
is path connected and does not contain $S$. Hence, (without loss of generality)
\[ \Lambda_2 \subset B_{c+7\delta + 3\sqrt{\delta (D+R)}}(y, 0). \]
But then
\[ \diam(\Lambda_2) \leq \diam(B_{c+7\delta + 3\sqrt{\delta (D+R)}}(y, 0)) \leq 2c+14\delta + 6\sqrt{\delta (D+R)}. \]
Thus \eqref{eq.assumption.delta} implies that
\[ \diam(\Lambda_2) < L/2 - 8\delta, \]
which contradicts the diameter lower bound. 
\end{proof}

\noindent \emph{Claim 2}. $B_{D+R}(\gamma(t_1)) \subset  B_{2\sqrt{b^2 + D^2}}(\gamma\mid_{|t-t_1| \leq D+R})$.
\begin{proof}
Let $\sigma$ denote the segment of $\gamma$ in $B_{D+R}(\gamma(t_1))$. Note that $\diam(\sigma) = 2D+2R$.

We first show that in $Y\times \R$, $B_{D+R}(y, 0) \subset B_{\sqrt{b^2 + D^2}}(f(\sigma))$. Suppose for contradiction that there is a point $p \in B_{D+R}(y, 0)$ satisfying
\[ d(p, f(\sigma)) > \sqrt{b^2 + D^2}. \]
Let $\pi$ denote the projection to $\R$ in $Y\times \R$. Then
\[ d_{\R}(\pi(p), \pi(f(\sigma))) > D. \]
Since $D > \delta$, $f$ is a $\delta$-approximation, and $\sigma$ is connected, we have
\[ \diam_{\R}(\pi(f(\sigma))) \leq 2R + D. \]
Then
\[ \diam(f(\sigma)) \leq \sqrt{b^2 + (2R+D)^2}. \]
By \eqref{eq.assumption.R} and $\delta<1$, we have
\[ \diam(\sigma) \leq \sqrt{b^2 + (2R+D)^2} + \delta < 2D + 2R, \]
which yields a contradiction.

Take $z \in B_{D+R}(\gamma(t_1))$. By the above, there is an $z' \in \sigma$ with $d(f(z), f(z')) \leq \sqrt{b^2 + D^2}$. Then
\[ d(z, \sigma) \leq d(z, z') \leq d(f(z), f(z')) + \delta \leq \sqrt{b^2 + D^2} + \delta. \]
The claim follows.
\end{proof}

\noindent \emph{Claim 3}. $\overline{E}_1 \setminus E_2 \subset B_{A}(\gamma\mid_{|t-t_1| < D+R})$.
\begin{proof}
Let $x' \in \overline{E}_1 \setminus E_2$.

If $d(x', x) \leq R + c + r$, then
\begin{align*}
d(x', \gamma(t_1))
& \leq d(x', \Sigma_0) + \diam(\Sigma_0) + d(\gamma(t_0), \gamma(t_1))\\
& \leq (R + c) + c + 4L\\
& = R + D.
\end{align*}
Then by Claim 2, we have $d(x', \gamma\mid_{|t-t_1| < D+R}) \leq A$.

Now, suppose for contradiction that $d(x', x) > R + c + r$. Take the radial geodesic $\mu$ from $x$ to $x'$, and let $x''$ be the point on $\mu$ with
\[ d(x'', x) = R + c + r. \]
By the above observation (since $R > L$, we still have $x'' \in \overline{E}_1 \setminus E_2$), we have
\[ d(x'', \gamma\mid_{|t-t_1| < D+R}) \leq A. \]
However, since $\partial B_{r + a_1 + a_2}(x)$ separates $x''$ from $\gamma$ (by the definition of $E_2$), we have
\[ d(x'', \gamma\mid_{|t-t_1| < D+R}) \geq R + c - a_1 - a_2 \geq R + c - 4L. \]
Since $R \geq A + 4L$, we reach a contradiction.
\end{proof}

By Claim 3, the diameter of $\overline{E}_1 \setminus E_2$ is bounded from above by $2A + 2D + 2R$ (which is bounded by a universal constant). Since $\Ric_g\geq0$, the Bishop--Gromov inequality yields a universal constant $C$ so that
\[ \vol(\overline{E}_1 \setminus E_2) \leq C, \]
as desired.
\end{proof}

We can now prove the main result. 
\begin{proof}[Proof of Theorem \ref{thm:main}]
We assume $M$ has one end. Let $r_0(x_0, M, g)$ and $C_1$ be the constants in Lemma \ref{lem:main}, where we assume $K \subset B_{r_0}(x_0)$.

First, assume (\ref{eq:scal_bound}). Take $r > 0$ very large (so that $r^{1-\alpha/2} > r_0$). Set $\tilde{g} := r^{-\alpha}g$. Then
\[ R_{\tilde{g}} \geq 1\ \ \text{on}\ \ B_{r^{1-\alpha/2}}^{\tilde{g}}(x_0). \]
Let $k \in \N$ so that $r_0 \leq r^{1-\alpha/2} - kL < r_0 + L$. Set
\[ r_i := r^{1-\alpha/2} - kL + iL. \]
Let $E_i$ be the unique unbounded component of $M \setminus \overline{B}_{r_i}^{\tilde{g}}(x_0)$. By Lemma \ref{lem:main} (with $a_1 = a_2 = L$), we have
\[ \vol_{\tilde{g}}(\overline{E}_i \setminus E_{i+1}) \leq C_1 \]
for $1 \leq i \leq k-1$, so
\[ \vol_g(\overline{E}_i \setminus E_{i+1}) \leq C_1r^{3\alpha/2}. \]
Note that
\[ B_r^g(x_0) = B_{r^{1-\alpha/2}}^{\tilde{g}}(x_0) \subset (M \setminus E_1) \cup \bigcup_{i=1}^{k-1} (\overline{E}_i \setminus E_{i+1}). \]
Moreover, by the choice of $k$, there is a constant $V(M, g) > 0$ so that
\[ \vol_g(M \setminus E_1) \leq V. \]
Hence, we have (since $k \leq r^{1-\alpha/2}/L$)
\[ \vol_g(B_r(x_0)) \leq V + C_1kr^{3\alpha/2} \leq V + \frac{C_1}{L}r^{1 + \alpha}. \]
Then we have
\[ \lim_{r \to \infty} \frac{1}{r^{1+\alpha}} \vol_g(B_r(x_0)) < +\infty \]
and (since $\alpha < 2$)
\[ \lim_{r \to 0} \frac{1}{r^{1+\alpha}} \vol_g(B_r(x_0)) = 0, \]
so the conclusion holds.

Now, assume for contradiction that (\ref{eq:scal_growth}) fails at some $x_0 \in M$. Let
\[ f_1(r) := \inf_{M \setminus B_r(x_0)} R_g. \]
By construction, $f_1$ is nonnegative and increasing. By the contradiction assumption, we have
\[ \lim_{r \to \infty} f_1(r) = + \infty. \]
Assuming $r_0$ sufficiently large, we have $f(r_0) \geq 1$. On $[r_0, \infty)$, we define a function $f$ to be the largest nondecreasing, piecewise constant function taking the values $\{4^jf_1(r_0)\} \to \infty$ so that the preimage of each value has length at least 1 and $f \leq f_1$. Then $f$ is nonnegative, increasing, has
\[ \lim_{r \to \infty} f(r) = + \infty, \]
and satisfies
\[ f(r+1) \leq 4f(r)\ \ \forall\ \ r \geq r_0'. \]
Moreover, $R_g(x) \geq f_1(d(x_0, x)) \geq f(d(x_0, x))$. Starting with $r_0$, we inductively define
\[ r_i := r_{i-1}+2Lf(r_{i-1})^{-1/2}. \]
Suppose for contradiction that $r_i \leq N < +\infty$ for all $i$. Then
\[ r_i - r_{i-1} = 2Lf(r_{i-1})^{-1/2} \geq 2Lf(N)^{-1/2} > 0, \]
which yields a contradiction. Hence, $r_i \to +\infty$. We also have,
\[ r_i - r_{i-1} = 2Lf(r_{i-1})^{-1/2} \leq 1 \]
(by assuming $r_0$ sufficiently large).
Let $E_i$ be the unique unbounded component of $M \setminus \overline{B}_{r_i}(x_0)$.
Let $g_i := f(r_{i-1})g$. Then
\[ R_{g_i} \geq 1 \ \ \text{on}\ \ M \setminus B_{r_{i-1}f(r_{i-1})^{1/2}}^{g_i}(x). \]
We first note that $r_{i-1}f(r_{i-1})^{1/2} \geq r_0$. Moreover, we have
\[ d_{g_i}(\partial E_{i-1}, \partial E_i) = f(r_{i-1})^{1/2}(r_i - r_{i-1}) = 2Lf(r_{i-1})^{-1/2}f(r_{i-1})^{1/2} = 2L  \]
and
\[ d_{g_i}(\partial E_i, \partial E_{i+1}) = f(r_{i-1})^{1/2}(r_{i+1} - r_i) = 2Lf(r_i)^{-1/2}f(r_{i-1})^{1/2} \in [L, 2L] \]
because
\begin{align*}
2L \geq 2Lf(r_i)^{-1/2}f(r_{i-1})^{1/2} & \geq 2Lf(r_i)^{-1/2}f(r_i-1)^{1/2}\\
& \geq Lf(r_i)^{-1/2}f(r_i)^{1/2} = L.
\end{align*}
Then by Lemma \ref{lem:main} (with $a_1$ and $a_2$ the distances above), we have
\[ \vol_{g_i}(\overline{E}_i \setminus E_{i+1}) \leq C_1, \]
so
\[ \vol_g(\overline{E}_i \setminus E_{i+1}) \leq C_1f(r_{i-1})^{-3/2}. \]
Then
\[ \vol_g(B_{r_k}(x_0)) \leq V + C_1\sum_{i=1}^{k-1} f(r_{i-1})^{-3/2}. \]
Since $r_i - r_{i-1} = 2Lf(r_{i-1})^{-1/2}$, we have
\begin{align*}
\sum_{i=1}^{k-1} f(r_{i-1})^{-3/2}
& = \frac{1}{2L}\sum_{i=1}^{k-1} (r_i - r_{i-1})f(r_{i-1})^{-1}.
\end{align*}
Then
\begin{align*}
\lim_{k \to \infty} \frac{1}{r_k}\vol_g(B_{r_k}(x_0)) \leq \frac{C_1}{2L} \lim_{k\to \infty} \frac{1}{r_k}\sum_{i=1}^k (r_i-r_{i-1})f(r_{i-1})^{-1}.
\end{align*}
Let $\eps > 0$. Let $k \in \N$ sufficiently large so that $f(r_k)^{-1} < \eps/2$. Let $l \in \N$ sufficiently large so that
\[ f(r_0)^{-1}\frac{r_k - r_0}{r_l} < \eps/2. \]
Then
\[ \frac{1}{r_l}\sum_{i=1}^l (r_i-r_{i-1})f(r_{i-1})^{-1} \leq f(r_0)^{-1}\frac{r_k-r_0}{r_l} + f(r_k)^{-1} < \eps. \]
Hence, we have
\[ \lim_{k \to \infty} \frac{1}{r_k}\vol_g(B_{r_k}(x_0)) = 0, \]
which contradicts Yau's linear volume growth (cf.\ \cite[Theorem 4.1]{SY:book}) since $M$ is noncompact.
\end{proof}

\section{Sharpness of main result}
We provide an example to demonstrate the sharpness of the growth upper bounds in Theorem \ref{thm:main}.

Consider on $[1, \infty) \times S^2$ the metric
\[ g = dt^2 + \rho(t)^2g_{S^2}^{\text{round}}. \]
We glue a compact cap so that positive Ricci curvature is preserved, so it suffices to study the scalar curvature decay and volume growth on this end.

Let $X_i \in T_pS^2$ be an orthonormal basis with respect to $g_{S^2}$. By \cite[\S 4.2.3]{Petersen:Riemannian}, we have
\[ R_{(t, p)}(\rho(t)^{-1}X_i, \partial_t, \partial_t, \rho(t)^{-1}X_i) = -\rho''(t)/\rho(t) \]
and
\[ R_{(t, p)}(\rho(t)^{-1}X_1, \rho(t)^{-1}X_2, \rho(t)^{-1}X_2, \rho(t)^{-1}X_1) = 1/\rho(t)^2. \]
Hence,
\begin{align*}
& \Ric_g(\partial_t, \partial_t) = -2\rho''(t)/\rho(t)\\
& \Ric_g(\rho(t)^{-1}X_i, \rho(t)^{-1}X_i) = -\rho''(t)/\rho(t) + 1/\rho(t)^2\\
& R_g = -5\rho''(t)/\rho(t) + 1/\rho(t)^2.
\end{align*}

Let $0 < \alpha < 2$. Take $\rho(t) = t^{\alpha/2}$. Then
\begin{align*}
& \rho(t) = t^{\alpha/2}\\
& \rho'(t) = \frac{\alpha}{2}t^{\alpha/2-1}\\
& \rho''(t) = \frac{\alpha}{2}\left(\frac{\alpha}{2}-1\right)t^{\alpha/2-2}.
\end{align*}
Hence, $\Ric_g > 0$ and
\[ R_g = -\frac{5\alpha}{2}\left(\frac{\alpha}{2}-1\right)t^{-2} + t^{-\alpha} \geq t^{-\alpha}. \]
Take $x_0 = (1, p_0)$. Then
\[ d_g((1, p_0), (t, p)) \leq t - 1 + d_{S^2}(p, p_0) \geq t - 1. \]
Then for $x = (t, p)$ with $d(x_0, x) \geq 1$, we have
\[ R_g(x) \geq t^{-\alpha} \geq (d_g(x_0, x) + 1)^{-\alpha} \geq 2^{-\alpha}d_g(x_0, x)^{-\alpha}. \]

Finally we compute the volume growth. For $r > 2\pi$, we have
\begin{align*}
\vol_g(B_r(x_0))
& \geq \vol_g([1, r/2] \times S^2)\\
& = C\int_1^{r/2} t^{\alpha} dt\\
& = \frac{C}{2^{1+\alpha}(1+\alpha)}r^{1+\alpha} - \frac{C}{1+\alpha}.
\end{align*}
Then
\[ \lim_{r \to \infty} r^{-1-\alpha}\vol_g(B_r(x_0)) \geq \frac{C}{2^{1+\alpha}(1+\alpha)} > 0. \]

See also the discussion in \cite{MW:2} after their statement of Theorem 1.4 for a related example demonstrating that it is possible to have nearly non-negative Ricci curvature $\Ric_g \geq -C d(x_0,x)^{-2}\log d(x_0,x)$ with $\inf_{M\setminus B_r(x_0)} R_g \to\infty $ as $r\to\infty$. 

\appendix

\section{Gromov-Hausdorff approximations}

We recall the definition of a Gromov-Hausdorff approximation.
\begin{definition}\label{defi:eps-GH-approx}
A map $f : (X, d_X) \to (Y, d_Y)$ is an \emph{$\eps$-Gromov-Hausdorff approximation} if
\[ |d_X(x_1, x_2) - d_Y(f(x_1), f(x_2))| < \eps \]
for all $x_1, x_2 \in X$ and
\[ Y \subset B_{\eps}(f(X)). \]
\end{definition}

We recall a notion of Gromov-Hausdorff distance between metric spaces using Gromov-Hausdorff approximations.
\begin{definition}\label{defi:GH}
We say $d_{\text{GH}}((X, d_X), (Y, d_Y)) < \eps$ if there are $\eps$-Gromov-Hausdorff approximations
\[ f: X \to Y \ \ \text{and} \ \ g : Y \to X. \]
\end{definition}

\begin{proposition}\label{prop:centers}
Let
\[ f : B_R(x) \subset (X, d_X) \to B_R(y, 0) \subset (Y \times \R, d_Y \times d_{\text{Euc}}) \]
be an $\eps$-Gromov-Hausdorff approximation. Then
\[ d(f(x), (y, 0)) < \sqrt{\eps}\sqrt{9\eps + 8R}. \]
\end{proposition}
\begin{proof}
Let $y' \in B_R(y, 0)$. By definition, there is an $x' \in B_R(x)$ with $d(f(x'), y') < \eps$. Moreover, $d(x, x') < R$, so $d(f(x), f(x')) < R + \eps$. Hence,
\[ d(f(x), y') \leq d(f(x), f(x')) + d(f(x'), y') < R + 2\eps, \]
which implies
\[ B_R(y, 0) \subset B_{R + 2\eps}(f(x)). \]

Write $f(x) = (y_0, t_0)$. Without loss of generality (by relabeling plus and minus), we have $t_0 \geq 0$. Since $(y, -R + \eps) \in B_R(y, 0)$, we have
\[ (R + 2\eps)^2 > d(f(x), (y, - R + \eps)) = d(y_0, y)^2 + (t_0 + R - \eps)^2. \]
Since $d(y_0, y)^2 \geq 0$, we have $t_0 \leq 3\eps$. Since $t_0 \geq 0 \geq -\eps$, we have
\[ d(y_0, y)^2 \leq (R + 2\eps)^2 - (R - 2\eps)^2 = 8R\eps. \]
The conclusion follows.
\end{proof}

\begin{proposition}\label{prop:prod_met}
Suppose $(X, d)$ is a path connected metric space. Take any $R > 0$ and $x \in X$, and let $(\tilde{X}, \tilde{d})$ the ball of radius $R > 0$ centered at $(x, 0)$ in the product metric space $(X \times \R, d \times d_{\text{Euc}})$. Let $r < \min\{\frac{1}{2}\diam(X), R\}$. Then the subset
\[ \tilde{X} \setminus B_r(x, 0) \subset \tilde{X} \]
is path connected.
\end{proposition}
\begin{proof}
We first show that the region $(\tilde{X} \setminus B_r(x, 0)) \cap (\tilde{X} \times \R_+)$ is path connected. Let $(x_i, t_i) \in \tilde{X} \setminus B_r(x, 0)$ for $i = 1, 2$ with $t_i \geq 0$. Let $\gamma(s) = (x(s), t(s))$ be any path in $\tilde{X}$ joining $(x_1, t_1)$ to $(x_2, t_2)$. By replacing $t(s)$ by $\max\{t(s), 0\}$, we obtain a continuous path in $(\tilde{X} \setminus B_r(x, 0)) \cap (\tilde{X} \times \R_+)$ joining the points, so we can assume $t(s) \geq 0$. Now we take
\[
\tilde{t}(s) := 
\begin{cases}
\sqrt{r^2 - d(x(s), x)^2} & d(x(s), x)^2 + t(s)^2 \leq r^2\\
t(s) & \text{otherwise}.
\end{cases}
\]
Since $t(s) \geq 0$, $\tilde{t}(s)$ is continuous. Moreover, if $\tilde{t}(s) \neq t(s)$, then
\[ d(x(s), x)^2 + \tilde{t}(s)^2 = d(x(s), x)^2 + r^2 - d(x(s), x)^2 = r^2, \]
so $\tilde{\gamma}(s) := (x(s), \tilde{t}(s))$ is a path in $(\tilde{X} \setminus B_r(x, 0)) \cap (\tilde{X} \times \R_+)$ joining the two points.

By the same argument for the $\R_-$ side, we have $(\tilde{X} \setminus B_r(x, 0)) \cap (\tilde{X} \times \R_+)$ is path connected.

Since $r < \frac{1}{2}\diam(X)$ and $X$ is path connected, there is an $x' \in X$ with $r < d(x', x) < R$. Then $(x', s)$ is a path in $\tilde{X} \setminus B_r(x, 0)$ for $s$ sufficiently small, which joins the two path connected regions. Hence, the conclusion follows.
\end{proof}

\bibliographystyle{amsalpha}
\bibliography{psc_vol_growth}

\providecommand{\bysame}{\leavevmode\hbox to3em{\hrulefill}\thinspace}
\providecommand{\MR}{\relax\ifhmode\unskip\space\fi MR }
\providecommand{\MRhref}[2]{%
  \href{http://www.ams.org/mathscinet-getitem?mr=#1}{#2}
}
\providecommand{\href}[2]{#2}
\begin{thebibliography}{MW22}

\bibitem[And90]{Anderson}
Michael~T. Anderson, \emph{Convergence and rigidity of manifolds under {R}icci
  curvature bounds}, Invent. Math. \textbf{102} (1990), no.~2, 429--445.
  \MR{1074481}

\bibitem[CC96]{CC:splitting}
Jeff Cheeger and Tobias~H. Colding, \emph{Lower bounds on {R}icci curvature and
  the almost rigidity of warped products}, Ann. of Math. (2) \textbf{144}
  (1996), no.~1, 189--237. \MR{1405949}

\bibitem[CG72]{CheegerGromoll}
Jeff Cheeger and Detlef Gromoll, \emph{On the structure of complete manifolds
  of nonnegative curvature}, Ann. of Math. (2) \textbf{96} (1972), 413--443.
  \MR{309010}

\bibitem[CL20]{CL:aspherical}
Otis Chodosh and Chao Li, \emph{Generalized soap bubbles and the topology of
  manifolds with positive scalar curvature},
  \url{https://arxiv.org/abs/2008.11888} (2020).

\bibitem[CL21]{CLstable}
Otis Chodosh and Chao Li, \emph{Stable minimal hypersurfaces in
  $\mathbf{R}^4$}, \url{https://arxiv.org/abs/2108.11462} (2021).

\bibitem[CL22]{CL:aniso}
Otis Chodosh and Chao Li, \emph{Stable anisotropic minimal hypersurfaces in
  $\mathbf{R}^4$}, \url{https://arxiv.org/abs/2206.06394} (2022).

\bibitem[CLS22]{CLS:stable}
Otis Chodosh, Chao Li, and Douglas Stryker, \emph{Complete stable minimal
  hypersurfaces in positively curved 4-manifolds},
  \url{https://arxiv.org/abs/2202.07708} (2022).

\bibitem[Gro86]{Gromov:large}
M.~Gromov, \emph{Large {R}iemannian manifolds}, Curvature and topology of
  {R}iemannian manifolds ({K}atata, 1985), Lecture Notes in Math., vol. 1201,
  Springer, Berlin, 1986, pp.~108--121. \MR{859578}

\bibitem[Gro18]{Gromov:metric-inequalities}
Misha Gromov, \emph{Metric inequalities with scalar curvature}, Geom. Funct.
  Anal. \textbf{28} (2018), no.~3, 645--726. \MR{3816521}

\bibitem[Liu13]{Liu:milnor}
Gang Liu, \emph{3-manifolds with nonnegative {R}icci curvature}, Invent. Math.
  \textbf{193} (2013), no.~2, 367--375. \MR{3090181}

\bibitem[MW21]{MW:1}
Ovidiu Munteanu and Jiaping Wang, \emph{Comparison theorems for
  three-dimensional manifolds with scalar curvature bound}, to appear in Int.\
  Math.\ Res.\ Not., \url{https://arxiv.org/abs/2105.12103} (2021).

\bibitem[MW22]{MW:2}
\bysame, \emph{Comparison theorems for {3D} manifolds with scalar curvature
  bound, {II}}, \url{https://arxiv.org/abs/2201.05595} (2022).

\bibitem[Nab20]{Naber:conj}
Aaron Naber, \emph{Conjectures and open questions on the structure and
  regularity of spaces with lower {R}icci curvature bounds}, SIGMA Symmetry
  Integrability Geom. Methods Appl. \textbf{16} (2020), Paper No. 104, 8.
  \MR{4164873}

\bibitem[Pet08]{Petrunin:upperbound}
A.~M. Petrunin, \emph{An upper bound for the curvature integral}, Algebra i
  Analiz \textbf{20} (2008), no.~2, 134--148. \MR{2423998}

\bibitem[Pet16]{Petersen:Riemannian}
Peter Petersen, \emph{Riemannian geometry}, third ed., Graduate Texts in
  Mathematics, vol. 171, Springer, Cham, 2016. \MR{3469435}

\bibitem[SY82]{SY:ric}
Richard Schoen and Shing-Tung Yau, \emph{Complete three-dimensional manifolds
  with positive {R}icci curvature and scalar curvature}, Seminar on
  {D}ifferential {G}eometry, Ann. of Math. Stud., vol. 102, Princeton Univ.
  Press, Princeton, N.J., 1982, pp.~209--228. \MR{645740}

\bibitem[SY94]{SY:book}
R.~Schoen and S.-T. Yau, \emph{Lectures on differential geometry}, Conference
  Proceedings and Lecture Notes in Geometry and Topology, I, International
  Press, Cambridge, MA, 1994, Lecture notes prepared by Wei Yue Ding, Kung
  Ching Chang [Gong Qing Zhang], Jia Qing Zhong and Yi Chao Xu, Translated from
  the Chinese by Ding and S. Y. Cheng, With a preface translated from the
  Chinese by Kaising Tso. \MR{1333601}

\bibitem[Xu20]{Xu:Yauconj}
Guoyi Xu, \emph{Integral of scalar curvature on non-parabolic manifolds}, J.
  Geom. Anal. \textbf{30} (2020), no.~1, 901--909. \MR{4058542}

\bibitem[Yau92]{Yau:problems}
Shing-Tung Yau, \emph{Open problems in geometry}, Chern---a great geometer of
  the twentieth century, Int. Press, Hong Kong, 1992, pp.~275--319.
  \MR{1201369}

\bibitem[Zhu22]{Zhu:psc.int}
Bo~Zhu, \emph{Geometry of positive scalar curvature on complete manifold},
  \url{https://arxiv.org/abs/2201.12668} (2022).

\end{thebibliography}

\end{document}